\documentclass[11pt,a4paper,reqno]{amsart}

\usepackage{amsmath}
\usepackage{amsfonts}
\usepackage{a4wide}
\usepackage{amssymb}
\usepackage{amsthm}

\theoremstyle{plain}

\newtheorem{teo}{Theorem}[section]
\newtheorem{lemma}[teo]{Lemma}
\newtheorem{prop}[teo]{Proposition}
\newtheorem{cor}[teo]{Corollary}

\theoremstyle{definition}

\theoremstyle{remark}

\newtheorem{rem}[teo]{Remark}

\numberwithin{equation}{section}

\def\Ric{{\mathrm {Ric}}}

\def\div{\operatornamewithlimits{div}\nolimits}

\title[Bach-flat gradient steady Ricci solitons]{Bach-flat gradient steady Ricci solitons}
\date{\today}

\author[H.-D. Cao]{Huai-Dong Cao}
\address[Huai-Dong Cao]{Department of Mathematics\\ Lehigh University\\
Bethlehem, PA 18015} 
\email{huc2@lehigh.edu}

\author[G. Catino]{Giovanni Catino}
\address[Giovanni Catino]{SISSA -- International School for Advanced Studies, Via Bonomea 265, Trieste, Italy, 34136}
\email{catino@sissa.it}

\author[Q. Chen]{Qiang Chen}
\address[Qiang Chen]{Department of Mathematics\\ Lehigh University\\
Bethlehem, PA 18015} 
\email{qic208@lehigh.edu}

\author[C. Mantegazza]{Carlo Mantegazza}
\address[Carlo Mantegazza]{Scuola Normale Superiore di Pisa, P.za Cavalieri 7, Pisa, Italy, 56126}
\email{c.mantegazza@sns.it}

\author[L. Mazzieri]{Lorenzo Mazzieri}
\address[Lorenzo Mazzieri]{Scuola Normale Superiore di Pisa, P.za Cavalieri 7, Pisa, Italy, 56126}
\email{l.mazzieri@sns.it}

\begin{document}

\begin{abstract} In this paper we prove that any $n$-dimensional ($n\ge 4$) complete Bach-flat gradient steady Ricci soliton with 
positive Ricci curvature is isometric to the Bryant soliton. We also show that a three-dimensional gradient steady Ricci soliton with 
divergence-free Bach tensor is either flat or isometric to the Bryant soliton. In particular, these results improve the corresponding
classification theorems for complete locally conformally flat gradient steady Ricci solitons in \cite{caochen, mancat1}. 
\end{abstract}

\maketitle

\section{The results} 

A complete Riemannian metric $g_{ij}$ on a smooth manifold $M^n$
is called a {\it gradient Ricci soliton} if there exists a smooth
function $f$ on $M^n$ such that the Ricci tensor $R_{ij}$ of the
metric $g_{ij}$ satisfies the equation
$$R_{ij}+\nabla_i\nabla_jf=\rho \,g_{ij}$$
for some constant $\rho$. For $\rho=0$ the Ricci soliton is {\it
steady}, for $\rho>0$ it is {\it shrinking} and for $\rho<0$ {\it
expanding}. The function $f$ is called a {\it potential function}
of the gradient Ricci soliton. Clearly, when $f$ is a constant a
gradient Ricci soliton is simply a Einstein manifold. Thus Ricci
solitons are natural extensions of Einstein metrics.  Gradient
Ricci solitons play an important role in Hamilton's Ricci flow as
they correspond to self-similar solutions, and often arise as
singularity models. Therefore it is important to classify gradient
Ricci solitons or understand their geometry.

In this paper we shall focus our attention on  gradient steady
Ricci solitons $(M^n, g_{Ij}, f)$,  which are possible Type II singularity models in
the Ricci flow, satisfying the steady soliton equation 
\begin{equation}\label{steady}
R_{ij}+\nabla_i\nabla_jf=0.
\end{equation}

\medskip
It is now well-known that compact gradient steady solitons must be Ricci flat. 
In dimension $n=2$, Hamilton \cite{hamilton5} discovered the first example of a complete 
noncompact gradient steady soliton, defined on $\mathbb R^2$ and called the {\it cigar soliton}, 
where the metric is given explicitly by
$$ ds^2=\frac{dx^2 +dy^2}{1+x^2+y^2}.$$ 
The cigar soliton has positive curvature and is asymptotic to a cylinder of finite
circumference at infinity.  Furthermore, Hamilton  \cite{hamilton5} showed that {\sl the only
complete steady soliton on a two-dimensional manifold with
bounded (scalar) curvature $R$ which assumes its maximum
at an origin is, up to scaling,  the cigar soliton}. For $n\geq 3$, Robert Bryant proved that
{\sl there exists, up to scaling, a unique complete rotationally symmetric gradient Ricci
soliton on $\Bbb R^n$} (see, e.g.,  Chow et al. \cite{chowbookI}
for details). The Bryant soliton has positive sectional curvature, linear curvature decay,
and volume growth of geodesic balls of radius $r$ on the order of $r^{(n+1)/2}$. In the K\"ahler case,
the first author \cite{cao1} constructed a complete gradient steady K\"ahler-Ricci soliton on $\mathbb{C}^m$,
for $m\geq 2$, with positive sectional curvature and $U(m)$ symmetry. For additional examples, we refer the readers to the survey paper \cite{cao2} 
by the first author and the references therein. 

A well-known conjecture of Perelman \cite{perel1}, concerning gradient steady Ricci solitons, states that in dimension $n=3$ the Bryant 
soliton is the only complete noncompact ($\kappa$-noncollapsed) gradient steady soliton with positive curvature. Despite some recent important 
progresses, it remains a big challenge to prove this conjecture of Perelman. For $n\ge 4$, such a uniqueness result is not expected to hold, and it 
is desirable to find geometrically interesting conditions under which the uniqueness would hold. In \cite{caochen}, the first and third author proved 
that a complete noncompact $n$-dimensional ($n\ge 3$) locally conformally flat gradient  steady Ricci soliton with 
positive sectional curvature is isometric to the Bryant soliton. Moreover, they showed that a complete noncompact $n$-dimensional  locally conformally 
flat gradient  steady Ricci soliton is either flat  or isometric to the Bryant soliton. The same results for $n\ge 4$ were proved independently by the second 
and fourth author \cite{mancat1} by using different method.  More recently,  Brendle \cite{brendle2} (see also Proposition~\ref{bre} below) 
showed that for an $3$-dimensional gradient steady soliton $(M^3, g_{ij}, f)$ if the scalar curvature $R$ is positive and tends to zero at infinity, 
and that  $(M^3, g_{ij}, f)$ is asymptotic to the Bryant soliton in some suitable sense, then $(M^3, g_{ij}, f)$ is locally conformally flat, hence 
isometric to the Bryant soliton. When $n=4$, X. Chen and Y. Wang~\cite{chenwang} have proved that any $4$-dimensional complete half-conformally 
flat gradient steady Ricci soliton is either Ricci flat,  or locally conformally flat (hence isometric to the Bryant Soliton by~\cite {caochen} and~\cite{mancat1}).  
 
In this paper, motivated by the very recent work \cite{caochen2} on Bach-flat shrinking Ricci solitons, we study complete Bach-flat steady Ricci solitons.
A well-known fact is that if a $n$-dimensional manifold ($n\ge 4$) is either Einstein or locally conformally flat, then it is Bach-flat.  
In addition, in dimension $n=4$, if a 4-manifold is  {\it half-conformally flat } or {\it locally conformal to an Einstein}  4-manifold, then it is also Bach-flat. 

Let us recall that  on any $n$-dimensional manifold $(M^n, g_{ij})$ ($n\ge 4$) the Bach tensor,  introduced by R. Bach \cite{bach} in early 1920s' to study 
conformal relativity,  is defined by
\begin{equation}\label{bach}
B_{ij} =\frac 1 {n-3}\nabla^k\nabla^l W_{ikjl}+\frac 1 {n-2}
R_{kl}W{_i}{^k}{_j}^l.
\end{equation}
Here $W_{ikjl}$ is the Weyl tensor.  In terms of the Cotton tensor
 $$
 C_{ijk}=\nabla_i R_{jk}-\nabla_j R_{ik}-\frac {1}{2(n-1)} ( g_{jk} \nabla_i R - g_{ik} \nabla_j R)\,,
 $$
 we also have
\begin{equation}\label{bach2}
B_{ij}=\frac 1 {n-2} (\nabla_k C_{kij} + R_{kl}W{_i}{^k}{_j}^l)\,.
\end{equation}

Our first main result concerns the classification of Bach-flat gradient steady Ricci solitons:

\begin{teo}\label{teo1} For $n\geq 4$, let $(M^{n},g_{ij}, f)$ be a complete gradient steady Ricci soliton with positive Ricci curvature such 
that the scalar curvature $R$ attains its maximum at some interior point. If in addition $(M^{n},g_{ij}, f) $ is Bach-flat, then it is 
isometric to the Bryant soliton up to a scaling factor.
\end{teo}

In dimension three we can prove a stronger result. To describe it, note that when $n=3$, while the expression of $B_{ij}$  in~\eqref{bach} 
is not well defined, the expression in~\eqref{bach2} makes perfect sense, so we can use it to define the Bach tensor in 3-D as
\begin{equation}\label{HC}
B_{ij}=\nabla_k C_{kij}\,.
\end{equation}
 
\begin{teo}\label{teo3D} Let $(M^{3},g_{ij}, f)$ be a three-dimensional complete gradient Ricci solitons with divergence-free Bach tensor (i.e., $\div B=0$). 
Then $(M^{3},g,f)$ is either Einstein or locally conformally flat.
\end{teo}

Using the 3-D classification of locally conformally flat gradient steady Ricci solitons (see~\cite{caochen}), we have:

\begin{cor} A complete three-dimensional gradient  steady Ricci soliton $(M^{3},g_{ij}, f)$ with divergence-free Bach tensor is either flat or 
isometric to the Bryant soliton (up to a scaling factor).
\end{cor}

\begin{rem} The assumption of Bach-flat or  divergence-free Bach is, at least a priori, weaker than that of locally conformally flat.  
Thus, Corollary 1.3 could be very helpful in proving Perelman's conjecture stated before. 

\end{rem}

\medskip
Finally, in Section 5, we present some applications and discuss Bach-flat gradient expanding solitons.  
\medskip

\noindent {\bf Acknowledements.} Part of the work was carried out while the first and the third authors were visiting the Mathematical Sciences Center of Tsinghua University in Beijing. 
They would like to thank the Center for their hospitality and support. The research of the first author was partially supported by NSF grant DMS-0909581.  

The second, the fourth and the fifth authors are partially supported by the Italian project FIRB--IDEAS ``Analysis and Beyond''.

\medskip

\section{Background material} 

In this section, we recall some background material needed in the proof of our main theorems. 

Recall that on any $n$-dimensional Riemannian
manifold $(M^n, g_{ij} )$ ($n\ge 3 $),  the Weyl curvature tensor
is given by

\begin{align*}
W_{ijkl}  = & R_{ijkl} - \frac{1}{n-2}(g_{ik}R_{jl}-g_{il}R_{jk}-g_{jk}R_{il}+g_{jl}R_{ik})\\
& + \frac{R}{(n-1)(n-2)} (g_{ik}g_{jl}-g_{il}g_{jk}), \\
\end{align*}
and  the Cotton tensor by
 $$C_{ijk}=\nabla_i R_{jk}-\nabla_j R_{ik}-\frac {1}{2(n-1)} ( g_{jk} \nabla_i R - g_{ik} \nabla_j R).$$
In terms of the Schouten tensor
\begin{equation}\label{schouten}
A_{ij}= R_{ij}- \frac{R}{2(n-1)}g_{ij}, 
\end{equation}   
we have 
\begin{equation}\label{WRA}
W_{ijkl}=R_{ijkl} - \frac{1}{n-2} (g_{ik}A_{jl}-g_{il}A_{jk}-g_{jk}A_{il}+g_{jl}A_{ik}), 
\end{equation}
and  
\begin{equation}\label{CA}
C_{ijk}=\nabla_i A_{jk}-\nabla_j A_{ik}.
\end{equation}
 
It is well known that, for $n=3$, $W_{ijkl}$ vanishes identically, while $C_{ijk}=0$ if and only if
$(M^3, g_{ij})$ is locally conformally flat; for $n\ge 4$,
$W_{ijkl}=0$ if and only if  $(M^n, g_{ij})$ is locally
conformally flat.  Moreover, for $n\ge 4$, the Cotton tensor
$C_{ijk}$ is, up to a constant factor, the divergence of the Weyl tensor:
\begin{equation}\label{DW}
C_{ijk}=-\frac{n-2}{n-3} \nabla_l W_{ijkl}.
\end{equation} 
We remark that  $C_{ijk}$ is skew-symmetric in the
first two indices and trace-free in any two indices:
\begin{equation}\label{SC}
C_{ijk}=-C_{jik} \quad \mbox{and} \quad g^{ij}C_{ijk}=g^{ik}C_{ijk}=0\,. 
\end{equation}
Moreover, for $n\ge 4$, the Bach tensor is defined by
$$
B_{ij}=\frac 1 {n-3}\nabla^k\nabla^l W_{ikjl}+\frac 1 {n-2}
R_{kl}W{_i}{^k}{_j}^l.  
$$
By~\eqref{DW}, we have an equivalent expression of the Bach tensor:
\begin{equation}\label{BC}
B_{ij}=\frac 1 {n-2} (\nabla_k C_{kij} + R_{kl}W{_i}{^k}{_j}^l). 
\end{equation}
Next we recall some  basic facts about complete gradient steady
Ricci solitons satisfying Eq. (1.1) .

\begin{lemma} {\bf (Hamilton \cite{hamilton9})} Let $(M^n, g_{ij}, f)$
be a complete gradient steady Ricci soliton satisfying Eq.
(1.1). Then we have
\begin{equation}\label{DR}
\nabla_iR=2R_{ij}\nabla_jf,  
\end{equation}
and
$$R+|\nabla f|^2=C_0 $$ for some constant $C_0$. Here $R$
denotes the scalar curvature.
\end{lemma}
\begin{lemma} Let $(M^n, g_{ij}, f)$ be a complete gradient steady soliton. Then it has nonnegative
scalar curvature $R\ge 0$.
\end{lemma}
Lemma 2.2 is a special case of a more general result of B.-L. Chen
\cite{chen2} which states that $R\ge 0$ for any ancient solution
to the Ricci flow. 

Note that, by Lemma 2.2, the constant $C_0$ in Lemma 2.1 must be positive for any non-trivial gradient steady soliton. Hence, by scaling the metric $g$, we can normalize it to be one so that
\begin{equation}\label{Sum}
R+|\nabla f|^2=1. 
\end{equation}

\begin{lemma} {\bf (Cao-Chen \cite{caochen})}\label{growth}  Let $(M^n, g_{ij}, f)$ be a complete noncompact
gradient steady soliton with positive Ricci curvature $Rc>0$. Assume the scalar curvature $R$ attains its maximum at some origin
$O$. Then, there exist some constants $0<c_1\leq 1$ and $c_2>0$  such that the potential function $f$ satisfies the estimates
\begin{equation}\label{potential}
c_1r(x)-c_2 \leq -f(x) \leq  r(x) + |f(O)|, 
\end{equation}
where $r(x)=d(O, x)$ is the distance function from $O$.
In particular, $f$ is a strictly concave exhaustion function achieving its maximum at the only critical point  $O$, and the underlying manifold $M^n$ is diffeomorphic to $\mathbb{R}^n$.
\end{lemma}

Finally, in the spirit of~\cite{caochen2}, we recall the covariant 3-tensor $D_{ijk}$, 
\begin{align*}
 D_{ijk} = & \frac{1}{n-2}(R_{jk} \nabla_i f- R_{ik} \nabla_j f) +\frac{1}{2(n-1)(n-2)} (g_{jk}\nabla_i R-g_{ik} \nabla_j R)\\
& + \frac{R}{(n-1)(n-2)} (g_{ik} \nabla_j f - g_{jk}\nabla_i f)\,,
\end{align*}
which was first introduced in \cite{caochen} and played the key role
in classifying locally conformally flat gradient steady Ricci solitons  \cite{caochen} and Bach-flat gradient shrinking Ricci solitons \cite{caochen2}.
Note that, $D_{ijk}$ has the same symmetry properties as the Cotton tensor: 
\begin{equation}\label{SD}
D_{ijk}=-D_{jik} \quad \mbox{and} \quad g^{ij}D_{ijk}=g^{ik}D_{ijk}=0\,. 
\end{equation}

\begin{lemma} {\bf (Cao-Chen \cite{caochen, caochen2})} 
Let $(M^n, g_{ij}, f)$ ($n\ge 3$) be a complete gradient steady soliton. Then $D_{ijk}$ is related to the Cotton tensor
$C_{ijk}$ and the Weyl tensor $W_{ijkl}$ by
\begin{equation}\label{DCW}
D_{ijk}=C_{ijk}+W_{ijkl}\nabla_l f \,.
\end{equation}
\end{lemma}

On the other hand,  for any gradient Ricci soliton, it turns out that the Bach tensor $B_{ij}$ can be expressed in terms of $D_{ijk}$ and the 
Cotton tensor $C_{ijk}$ (for the proof see~\cite{caochen2}): 
\begin{equation}\label{BD}
B_{ij} =-\frac 1{n-2}(\nabla_k D_{ikj}+\frac{n-3}{n-2}C_{jli}\nabla_l f)\,.
\end{equation}

Moreover, we recall that the norm of $D_{ijk}$ is linked to the geometry of level surfaces of the potential function
$f$ by the following: 

\begin{lemma}  {\bf (Cao-Chen \cite{caochen, caochen2})}  Let $(M^n, g_{ij}, f)$ ($n\ge 3$) be an $n$-dimensional
gradient steady Ricci soliton. Then, at any point $p\in
M^n$ where $\nabla f(p)\ne 0$, we have
\begin{equation}\label{D2}
|D_{ijk}|^2=\frac{2|\nabla f|^4}{(n-2)^2} |h_{ab}-\frac{H}{n-1}g_{ab}|^2 +\frac{1}{2(n-1)(n-2)}
|\nabla_a R|^2
\end{equation}
where $h_{ab}$ and $H$ are the second fundamental form and the mean curvature of the level surface
$\Sigma=\{f=f(p)\}$, and $g_{ab}$ is the induced metric on $\Sigma$.
\end{lemma}

Thus, the vanishing of $D_{ijk}$
implies the umbilicity of the level surfaces of the potential function as well as the constancy of the scalar curvature on them 
(see also Proposition~\ref{propCC} below).  For further details on the tensor $D_{ijk}$ we refer the interested reader to~\cite[Section 3]{caochen2}\,.

\section{Proof of Theorem~\ref{teo1}}

As in \cite{caochen2}, the first step in proving Theorem 1.1 is to show that, for noncompact steady gradient Ricci solitons with positive Ricci curvature, the 
Bach-flatness implies the vanishing of the 3-tensor $D_{ijk}$. 

\begin{lemma}\label{D0} Let $(M^{n},g_{ij}, f)$ ($n\ge 3$) be a complete Bach-flat gradient steady Ricci soliton with positive Ricci curvature such 
that the scalar curvature $R$ attains its maximum at some interior point. Then, $D_{ijk}=0$.
\end{lemma}
\begin{proof} Since $(M^{n},g_{ij}, f)$ ($n\ge 4$) has positive Ricci curvature and that  
the scalar curvature $R$ attains its maximum at some interior point $O$, by Lemma 2.3, 
there exist constants 
$c_{1},c_{2}>0$ such that
\begin{equation}\label{GF}
f(x) \leq - c_{1} r(x) + c_{2} \,,
\end{equation}
where $r(x)$ is the distance to the origin $O$. Moreover, since $R_{ij}>0$, from the well-known Bishop volume comparison theorem we know 
that $(M^{n},g_{ij}, f)$ has at most Euclidean volume growth, i.e., there exists a positive constant $C>0$, such that
\begin{equation}\label{GV}
Vol (B_{s}(O)) \leq C \,s^{n}\,,
\end{equation}
for any geodesic ball $B_{s}(O)$. By the definition of $D_{ijk}$ and using the identities ~\eqref{BD},~\eqref{SD} and~\eqref{SC}, it follows from the same argument as in \cite{caochen2} that 
\begin{eqnarray*}
\int_{B_{s}(O)} B_{ij} \nabla_{i} f \nabla_{j} f e^{f} \,dV_{g} &=& -\frac{1}{n-2} \int_{B_{s}(O)} \nabla_{k} D_{ikj} \nabla_{i} f \nabla_{j} f e^{f} \,dV_{g} \\
&=& -\frac{1}{2} \int_{B_{s}(O)} |D_{ijk}|^{2} e^{f} \, dV_{g} - \frac{1}{n-2} \int_{\partial B_{s}(O)} D_{ijk} \nabla_{i} f \nabla_{j} f e^{f} \nu_{k}\, d\sigma \,,
\end{eqnarray*}
where $\nu$ denotes the outward unit normal to $\partial B_{s}(O)$. Again, from the definition of $D_{ijk}$, it is easy to check that, for sufficiently large $s$, we have 
\begin{align*}
\big|\int_{\partial B_{s}(O)} D_{ijk} \nabla_{i} f \nabla_{j} f e^{f} \nu_{k}\, d\sigma \big| \, & \leq  \, C \int_{\partial B_{s}(O)} (|R_{ij}| + |R|)|\nabla f|^{3} e^{f} \, d\sigma \, \\
& \leq  2C \int_{\partial B_{s}(O)} e^{f} \, d\sigma \,,
\end{align*}
where we have used identity (\ref{Sum}) and the fact that $|R_{ij}|\leq R$ (since $g$ has positive Ricci curvature). 
By letting $s\rightarrow +\infty$ and using~\eqref{GF}-\eqref{GV}, 
we obtain
$$
0 \,= \, \int_{M} B_{ij} \nabla_{i} f \nabla_{j} f e^{f} \,dV_{g} \, = \, -\frac{1}{2} \int_{M} |D_{ijk}|^{2} e^{f} \, dV_{g} \,, $$
implying $D_{ijk}=0$.

\end{proof}

By Lemma 2.5, the vanishing of the tensor $D_{ijk}$ implies many rigidity properties about the geometry of the level surfaces of the 
potential function $f$:

\begin{prop}[Proposition 3.2 in~\cite{caochen2}]\label{propCC}
Let $(M^n, g, f)$ ($n\ge 3 $) be any complete gradient Ricci
soliton with  $D_{ijk}=0$, and let  $c$ be a regular value of $f$
and $\Sigma_c=\{f=c\}$ be the level surface of $f$. Set
$e_1=\nabla f /|\nabla f |$ and pick any orthonormal frame $e_2,
\cdots, e_n$ tangent to the level surface $\Sigma_c$, Then:

\medskip

(a) $|\nabla f|^2$ and the scalar curvature $R$ of $(M^n,
g_{ij}, f)$ are constant on $\Sigma_c$;

\smallskip
(b) $R_{1a}=0$ for any $a\ge 2$, hence $e_1=\nabla f /|\nabla f |$
is an eigenvector of $Rc$;

\smallskip

(c) the second fundamental form $h_{ab}$ of $\Sigma_c$ is of the
form $h_{ab}=\frac{H}{n-1} g_{ab}$;

\smallskip

(d) the mean curvature $H$ is constant on $\Sigma_c$;

\smallskip

(e) on $\Sigma_c$, the Ricci tensor of $(M^n, g_{ij}, f)$ either
has a unique eigenvalue $\lambda$, or has two distinct eigenvalues
$\lambda$ and $\mu$ of multiplicity $1$ and $n-1$ respectively. In
either case, $e_1=\nabla f /|\nabla f |$ is an eigenvector of
$\lambda$. Both $\lambda$ and $\mu$ are constant on $\Sigma_c$.
\end{prop}
 
Now we are in the position to complete the proof of Theorem~\ref{teo1}: 

\medskip {\em Proof of Theorem 1.1}. \ Let $(M^{n},g,f)$, $n\ge 4$, be a complete Bach-flat gradient steady Ricci solitons with positive Ricci curvature such that the scalar curvature 
$R$ attains its maximum at some interior point $O\in M$.  Then, by Lemma~\ref{growth} we know that $f$ is proper, strictly concave, has a unique critical point at $O$, and that $M^n$ is diffeomorphic to $\mathbb{R}^{n}$. On the other hand, by Lemma~\ref{D0}, we have that $D_{ijk}=0$. Therefore for $n=4$, from~\cite[Theorem 1.4]{caochen2} and the assumption of positive Ricci curvature, we conclude that $(M^{4},g_{ij},f)$ is isometric to the Bryant soliton up to a scaling factor.

From now on let us consider $n\geq 5$. First of all, on $M \setminus \{O\}$, the soliton metric $g_{ij}$ can be expressed as 
$$ds^2=\frac{1}{|\nabla f|^2} df^2+g_{ab}(f, \theta) d\theta^a d\theta^{b},$$
where  $(\theta^2, \cdots, \theta^n)$ is  any local coordinates system on the lever surface $\Sigma=\{f=f(p)\}$ at $p\in M \setminus \{O\}$. 
Note that, since $D_{ijk}=0$, $|\nabla f|^2$ depends only on $f$ by Proposition~\ref{propCC} (a). Hence, by a suitable change of variable, we can further express $g_{ij}$
as $$ ds^2=dr^2+g_{ab}(r, \theta) d\theta^a
d\theta^{b}\ , \quad 0<r<\infty \,.$$ Here $r(x)$ is the distance function from $O$. We remark that, by Lemma 2.3, $|f|(x)$ is proportional to $r(x)$.  \\

\noindent {\bf Claim 1}: For $r>0$, the induced metric $\bar g_{\Sigma_r}=g_{ab}(r, \theta) d\theta^a
d\theta^{b}$ on each level surface $\Sigma_r$ is Einstein. \\ 

Indeed, we have the following more general fact: 

\begin{lemma}\label{propE}
Let $(M^n, g, f)$ ($n\ge 4$) be a complete gradient Ricci soliton with $D_{ijk}=0$. Then each regular level surface $\Sigma$, with the induced 
metric  $\bar g_{\Sigma}$, is an Einstein manifold. 
\end{lemma}

\begin{proof} Let  $\{e_1, e_2, \cdots, e_n\}$ be any orthonormal frame, with
$e_1=\nabla f/|\nabla f|$ and $e_2, \cdots, e_n$ tangent to
$\Sigma$. Then, by the Gauss equation and Proposition~\ref{propCC} (c),  the sectional curvatures of $(\Sigma, g_{ab})$ are given by
$$
R^{\Sigma}_{abab} = R_{abab} +h_{aa}h_{bb}-h_{ab}^2= R_{abab} + \frac{H^2}{(n-1)^2}\,.
$$
Thus, the Ricci curvatures of  $(\Sigma, g_{ab})$ are 
$$R^{\Sigma}_{aa}=R_{aa}-R_{1a1a} +\frac{H^2}{n-1}.$$
On the other hand, by Corollary 5.1 in~ \cite{caochen2}, we know that $W_{1a1a}=0$. Thus, 
\begin{align*} 
R_{1a1a}= & \frac{1}{n-2} (R_{aa}+R_{11})-\frac{R}{(n-1)(n-2)}\\
= & -R_{aa} +\frac {R}{n-1}\,.
\end{align*} 
Hence,  it follows that 
$$R^{\Sigma}_{aa}= 2R_{aa} +\frac {H^2-R} {n-1}\,.$$
But, by Proposition~\ref{propCC}, $R, H$ and $\mu=R_{aa}$ are constant along $\Sigma$. This proves that $(\Sigma, g_{ab})$ has constant Ricci curvature. 

\end{proof}

\noindent {\bf Claim 2}:  On $M \setminus \{O\}$, the metric $g$ takes the form of a warped product metric: 
\begin{equation}\label{metric}
ds^2 = dr^{2} + w(r)^{2} {\bar g}_{E}\,, \quad r\in(0,+\infty) \,,
\end{equation}
where $w$ is some nonnegative smooth function on $M^{n}$ vanishing only at $O$, and $\bar g_{E}=\bar g_{\Sigma_1}$ is the Einstein 
metric defined on the level surface $\Sigma_1$. \\


Indeed, by identity (5.3) in~\cite{caochen} and Propositon~\ref{propCC}, we have 
$$ \frac{\partial } {\partial r} g_{ab}=-2 h_{ab}= \phi (r) g_{ab}\ ,$$ where $\phi (r)=-2H(r)/(n-1)$. Thus, it follows easily that 
$$g_{ab} (r, \theta)=e^{\Phi(r)} g_{ab} (1, \theta),$$ where 
$$
\Phi(r)=\int_{1}^{r} \phi(r) \, dr.
$$ 
This proves Claim 2. \\

By scaling, we can assume that 
\begin{equation}\label{ein}
\Ric_{{\bar g}_{E}}=(n-2) k\,{\bar g}_{E}, \quad \mbox{with} \quad  k=-1,0, 1.
\end{equation}
We shall see below that in fact  $k=1$, as we expected. \\

\noindent {\bf Claim 3}: We have

$$
\lim_{r\rightarrow 0} \frac{w(r)}{r}=1 \,.
$$

Clearly, $w(r)\to 0$ as $r\to 0$. On the other hand, on $M\setminus\{O\}$, the Ricci tensor and the scalar curvature of the metric $g$ in~\eqref{metric} take the form (see~\cite[Proposition~9.106]{besse}) 
$$
\Ric_{g} = -(n-1) \frac{w''}{w} dr\otimes dr + \big((n-2)(k-(w')^{2}) - w\,w''\big)\,{\bar g}_{E}\,,
$$ and
$$
R_{g} = -2(n-1)\frac{w''}{w} + \frac{(n-1)(n-2)}{w^{2}}\big(k-(w')^{2}\big)
$$ respectively.   Here we have used the Claim 1 and the normalization \eqref{ein}.

From the expression of the Ricci tensor above and the fact that $|Rc|\le 1$ on $M^n$, it is immediate to see that $w''/w$ must be bounded as $r\rightarrow 0$. 
Hence, from the above scalar curvature expression, 
it is easy to deduce the claim.  In particular, we can conclude that the Einstein constant $k=1$ for the metric ${\bar g}_{E}$. \\

\noindent {\bf Claim 4}:  ${\bar g}_{E}$ is equal to the standard round metric ${\bar g}_{\mathbb{S}^{n-1}}$ on the unit sphere $\mathbb{S}^{n-1}$. \\


This essentially follows from the previous claims and the elementary fact that infinitesimally the metric $g$ is approximately Euclidean near $O$. 
In fact, the standard expansion of the metric $g$ around $O$, written in any normal coordinates $(x^1, \cdots, x^n)$, gives

\begin{eqnarray*} 
g &=& (\delta_{ij}+ \sigma_{ij}(x))\, dx^{i}\otimes dx^{j} \\
&=& g_{\mathbb{R}^{n}}+ \sigma_{ij}\, dx^{i}\otimes dx^{j} \,,
\end{eqnarray*} 
where $\sigma_{ij}=\mathcal{O}(|x|^{2})$. To pass to polar coordinates, we write $x^{i} = r \phi^{i}(\theta^{1},\ldots,\theta^{n-1)})$, with $r\in(0,+\infty)$ and $(\theta^{1},\ldots,\theta^{n-1})$ being
local coordinates on $\mathbb{S}^{n-1}$. Notice that $|\phi^{1}|^{2}+\dots+|\phi^{n}|^{2}=1$ and $|x|=r$. Thus, one has \begin{eqnarray*}
g &=& (1+\sigma_{ij}\phi^{i}\phi^{j})dr\otimes dr + r\,\sigma_{ij}\frac{\partial \phi^{i}}{\partial \theta^{\alpha}} \phi^{j} dr \otimes d\theta^{\alpha} +r\,\sigma_{ij}\frac{\partial \phi^{j}}{\partial \theta^{\alpha}} \phi^{i} d\theta^{\alpha} \otimes dr + \\
&& +\,\big( r^{2} {\bar g}^{\mathbb{S}^{n-1}}_{\alpha\beta}+\,r^{2}\sigma_{ij}\frac{\partial \phi^{i}}{\partial \theta^{\alpha}}\frac{\partial \phi^{j}}{\partial \theta^{\beta}}\big) \,d\theta^{\alpha} \otimes d\theta^{\beta}\,,
\end{eqnarray*} 
with $\sigma_{ij}=\mathcal{O}(r^{2})$. Comparing with~\eqref{metric}, we see that $\sigma_{ij}\phi^{j}=0$ and  
$$  w^{2} (r) {\bar g}_{E} \,= \, r^{2} {\bar g}_{\mathbb{S}^{n-1}} +\,r^{2}\sigma_{ij}\frac{\partial \phi^{i}}{\partial \theta^{\alpha}}\frac{\partial \phi^{j}}{\partial \theta^{\beta}} \,d\theta^{\alpha} \otimes d\theta^{\beta}\,,  \quad r\in(0,+\infty) \,.$$
Now using the fact that $\sigma_{ij}=\mathcal{O}(r^{2})$ and Claim 3, and taking the limit as $r\rightarrow 0$, we obtain
$${\bar g}_{E} \,=\, {\bar g}_{\mathbb{S}^{n-1}} \,. $$ Therefore, on $M\setminus\{O\}$, we have 
$$ds^2 = dr^{2} + w(r)^{2} {\bar g}_{\mathbb{S}^{n-1}}\,, \quad r\in(0,+\infty) \,, $$
proving that the soliton metric $g$ is rotationally symmetric. Therefore, it follows that $(M^n, g, f)$ is the Bryant soliton, because we know that $M^n$ is 
diffeomorphic to $\mathbb {R}^n$ and the Bryant soliton is the only non-flat rotationally symmetric gradient steady soliton on $\mathbb{R}^n$ up to scaling.  
This completes the proof of Theorem~\ref{teo1}. 

\qed

\medskip

\section{Proof of Theorem~\ref{teo3D}}

In the special case $n=3$, we can show that divergence-free Bach tensor implies the vanishing of the Cotton tensor for all gradient Ricci solitons by a pointwise argument, 
which allows us to remove the assumptions on the positivity of the Ricci curvature and the scalar curvature achieving its interior maximum. 

\medskip

\noindent {\em Proof of Theorem~\ref{teo3D}}. Let $(M^3, g, f)$ be a three-dimensional complete gradient Ricci soliton with divergence-free Bach tensor. We recall that in dimension three we have defined the Bach tensor as
\begin{equation}\label{BW}
B_{ij}=\nabla_k C_{kij}\,.
\end{equation}
We claim that 
\begin{equation}\label{divB}
\nabla_j B_{ij}=-C_{ijk}R_{jk}\,.
\end{equation}
Indeed,  in terms of the Schouten tensor 
$$ 
A_{ij}= R_{ij}- \frac{R}{4}g_{ij}\,,
$$  
and the Cotton tensor  
$$
C_{ijk}=\nabla_i A_{jk}-\nabla_j A_{ik}\,,
$$
we have 
$$
B_{ij}= \nabla_k (\nabla_kA_{ij}-\nabla_iA_{kj}) \,.
$$
Hence 
\begin{align*}
\nabla_i B_{ij}= & \nabla_i \nabla_k (\nabla_kA_{ij}-\nabla_iA_{kj}) \\
= & (\nabla_i \nabla_k-\nabla_k \nabla_i)\nabla_kA_{ij}\\
 = & -R_{il}\nabla_lA_{ij} +R_{kl}\nabla_kA_{lj}+R_{ikjl}\nabla_kA_{il}\\
= & R_{ikjl}\nabla_kA_{il}\,.
\end{align*}
On the other hand,  since the Weyl tensor $W=0$ in dimension three, (\ref{WRA}) becomes 
$$R_{ijkl} = g_{ik}A_{jl}-g_{il}A_{jk}-g_{jk}A_{il}+g_{jl}A_{ik}\,.$$
Therefore, 
$$
\nabla_i B_{ij}= (A_{jk}g_{il}C_{lki} +A_{ik}C_{kji}) = -R_{ki}C_{jki}\,,
$$
proving the claim.

Now assume $(M^3, g, f)$ is any three-dimensional gradient Ricci soliton. Recall that, for $n=3$, we have
\begin{align*} 
C_{ijk}= & D_{ijk} \\
= & R_{jk} \nabla_i f- R_{ik} \nabla_j f +\frac{1}{4} (g_{jk}\nabla_i R-g_{ik} \nabla_j R)
               - \frac{R}{2} (g_{jk}\nabla_i f  - g_{ik} \nabla_j f )\,.
\end{align*}
Thus,  using~\eqref{divB} and~\eqref{SC}, we get 
$$ \div(B)\cdot \nabla f=-C_{ijk}R_{jk}\nabla_i f=-\frac{1}{2} |C_{ijk}|^2.$$
Therefore, $\div (B)=0$ implies the Cotton tensor $C_{ijk}=0$, which is equivalent to that $(M^3, g, f)$ is locally conformally flat. 

\qed

Consequently, by combining Theorem~\ref{teo3D} and the classification theorem in~\cite{caochen} for three-dimensional complete locally 
conformally flat gradient steady Ricci solitons, we have 

\begin{cor} Let $(M^3, g, f)$ be a complete gradient  steady Ricci soliton with divergence-free Bach tensor, then it is either flat or isometric to the Bryant soliton.
\end{cor}

\begin{rem}\label{divbach}  For $n\ge 4$, it is known among experts that the divergence of the Bach tensor is given by
$$
 \nabla_j B_{ij}=\frac{n-4}{(n-2)^2} C_{ijk}R_{jk}\,.
$$
\end{rem}

\section{Further remarks}

It was proved in~\cite[Theorem 1.4]{caochen2} that any 4-dimensional complete gradient steady Ricci soliton with $D_{ijk}=0$ is either Ricci flat, or locally conformally flat but non-flat (hence isometric 
to the Bryant soliton by~\cite{caochen} and~\cite{mancat1}). In the proof of Theorem~\ref{teo1}, we have actually shown the following 

\begin{prop}\label{prop27} Let $(M^{n},g_{ij}, f)$, $n\geq 4$, be a complete gradient steady Ricci soliton with $D_{ijk}=0$. If, in addition, the Ricci curvature is positive and the scalar curvature $R$ attains its maximum at some interior point, then $(M^{n},g_{ij}, f)$ is isometric to the Bryant soliton up to a scaling factor.
\end{prop}

On the other hand, Brendle~\cite{brendle2} proved the following result\footnote{Although Brendle only stated this result for $n=3$ in \cite{brendle2}, the same 
argument, as shown by him in the preprint arXiv:1010.3684v1,  works for all dimensions $n\ge 4$.}\,: 
 
\begin{prop}[\bf{Brendle~\cite{brendle2}}]\label{bre}
Let $(M^n, g, f)$ ($n\ge 3$) be a $n$-dimensional gradient steady Ricci soliton. Suppose that the 
scalar curvature R of $(M^n, g)$ is positive and approaches zero at infinity. Denote by $\psi: (0,1) \to \mathbb R$ the smooth function such that the vector field 
$$X=:\nabla R +\psi (R)\nabla f=0$$ on the Bryant soliton, and define $u: (0,1) \to \mathbb R$ by
$$ u(s)=\log \psi (s) +\frac{1}{n-1}\int_{1/2}^{s} (\frac{n}{1-t} - \frac{n-1-(n-3)t}{(1-t)\psi(t)}) dt. $$
Moreover,  assume that there exists an exhaustion of $M^n$ by bounded domains $\Omega_l$ 
such that
\begin{equation}\label{asym}
\lim_{l \to \infty}  \int_{\partial\Omega_l} e^{u(R)} <\nabla R +\psi (R) \nabla f, \nu> \, = 0. 
\end{equation}
Then $X=0$ and $D_{ijk}=0$. In particular, for $n=3$, $(M^3, g, f)$ is isometric to the Bryant soliton. 

\end{prop}

As an immediate consequence of Proposition~\ref{prop27} and Proposition~\ref{bre}, we obtain

\begin{cor}  Let $(M^{n},g_{ij}, f)$ ($n\ge 4$) be a complete gradient steady Ricci soliton with positive Ricci curvature such 
that the scalar curvature $R$ approaches zero at infinity. Moreover,  assume that condition~\eqref{asym} in Proposition~\ref{bre} is 
satisfied for some exhaustion of $M^n$ by bounded domains $\Omega_l$. Then $(M^n, g, f)$ is isometric to the Bryant soliton.
\end{cor}

\begin{rem} By Lemma~\ref{growth},  $f$ is an exhaustion function on $M^n$. 
\end{rem}

Finally, the techniques used in the proof of Theorem~\ref{teo1} can be easily adapted to the case of complete gradient expanding Ricci solitons with nonnegative Ricci curvature 
which are solutions of the equation
\begin{equation}\label{expanding}
R_{ij}+\nabla_i\nabla_jf=-\tfrac{1}{2} \,g_{ij}\,.
\end{equation}
We also normalize the potential function $f$, up to an additive constant, by
\begin{equation}\label{Df}
R+|\nabla f|^2+f=0,
\end{equation}
which is a well-known identity for expanding Ricci solitons (see~\cite{hamilton9}).

The need ingredient is the following lemma, which should be known to experts in the field:

\begin{lemma}\label{growth2} Let $(M^n, g_{ij}, f)$ ($n\ge 3$) be a complete noncompact gradient expanding soliton with nonnegative Ricci curvature $Rc\ge 0$. 
Then, there exist some constants $c_1>0$ and $c_2>0$  such that the potential function $f$ satisfies the estimates
\begin{equation}\label{potential2}
 \frac{1}{4}\big(r(x) -c_1\,\big)^{2}-c_2 \leq -f(x) \leq  \frac{1}{4}\big(r(x) +2\sqrt{-f(O)}\,\big)^{2}\,, 
\end{equation}
where $r(x)$ is the distance function from any fixed base point in $M^n$. In particular, $f$ is a strictly concave exhaustion function achieving its maximum at 
some interior point $O$, which we take as the base point,  and the underlying manifold $M^n$ is diffeomorphic to $\mathbb{R}^n$.
\end{lemma}

 \begin{proof} 
The upper bound follows from~\eqref{Df} and the assumption of $R\ge 0$ which together imply $|\nabla (-f)|^2\le (-f)$.  The lower bound 
is an easy consequence of the second variation of arc length argument as in,  e.g., ~\cite[p.179]{caozhou}, applied to the equation 
$$\nabla_{i}\nabla_j (-f)=R_{ij}+\frac{1}{2}g_{ij}\ge \frac{1}{2}g_{ij}.$$ 
Moreover, since $ |\nabla (-f)| \leq  \sqrt{-f} \leq \frac{1}{2} r(x) +\sqrt{-f(O)}$, $-f$ is clearly proper and hence
an exhaustion function. Therefore $M^n$ is diffeomorphic to $\mathbb{R}^n$.  

\end{proof}

\begin{rem} Clearly, in Lemma~\ref{growth2} and the results below, we can replace the assumption of nonnegative Ricci curvature $Rc\ge 0$ 
by $Rc\ge -(\frac 1 2-\epsilon)g$ for any small $\epsilon >0$. Of course, the normalizing of $f$ and the coefficients in \eqref{potential2} has
to be adjusted accordingly. 
\end{rem}

Taking advantage of this growth estimates on the potential function $f$, it is immediate to deduce the analogous of Lemma~\ref{D0} for expanding solitons, namely

\begin{lemma}\label{E0} Let $(M^{n},g_{ij}, f)$ ($n\ge 3$) be a complete Bach-flat gradient expanding Ricci soliton with nonnegative Ricci curvature. 
Then, $D_{ijk}=0$.
\end{lemma}

Having this at hand, it is sufficient to follow the {\em proof of Theorem~\ref{teo1}} in Section 3 to obtain the rotational symmetry. More precisely, we have

\begin{teo} For $n\geq 4$, let $(M^{n},g_{ij}, f)$ be a complete Bach-flat gradient expanding Ricci soliton with nonnegative Ricci curvature, then it is 
rotationally symmetric. 
\end{teo}

For $n=3$, by using Theorem~\ref{teo3D}, we have 
\begin{teo} Let $(M^{3},g_{ij}, f)$ be a three-dimensional complete expanding gradient Ricci solitons with divergence-free Bach tensor and nonnegative 
Ricci curvature. Then $(M^{3},g,f)$ is rotationally symmetric.
\end{teo}

For a discussion of the expanding Ricci solitons which are rotationally symmetric, see~\cite[Chapter~1, Section~5]{chowbookI}, where
the authors provide the existence of solutions with positive Ricci curvature (analogous to the Bryant soliton).


\

\bibliographystyle{amsplain}
\bibliography{bachsteady}

\

\
\end{document}